\documentclass[10pt]{article}
\usepackage{amsmath,amsthm,amssymb,amsfonts}
\usepackage{enumerate}
\usepackage{mathrsfs}
\usepackage[cmtip,all]{xy}
\numberwithin{equation}{section}
\numberwithin{figure}{section}

\normalsize

\newtheoremstyle{theoremstyle}
  {10pt}      
  {5pt}       
  {\itshape}  
  {}          
  {\bfseries} 
  {:}         
  {.5em}      
  {}          

\newtheoremstyle{examplestyle}
  {10pt}      
  {5pt}       
  {}          
  {}          
  {\bfseries} 
  {:}         
  {.5em}      
  {}          

\newtheorem{thm}{Theorem}[section]
\newtheorem{prop}[thm]{Proposition}
\newtheorem{lem}[thm]{Lemma}
\newtheorem{cor}[thm]{Corollary}

\theoremstyle{definition}

\newtheorem{defn}[thm]{Definition}
\newtheorem{remk}[thm]{Remark}
\newtheorem{remks}[thm]{Remarks}

\newtheorem{exm}[thm]{Example}
\newtheorem{exms}[thm]{Examples}

\numberwithin{equation}{section}

{\hfill$\square$\end{defn}}

{\hfill$\square$\end{remk}}

{\hfill$\square$\end{remks}}

{\hfill$\square$\end{exm}}

{\hfill$\square$\end{exms}}

\newcommand{\thmref}{Theorem~\ref}
\newcommand{\propref}{Proposition~\ref}

\newcommand{\sC}{{\mathcal C}}
\newcommand{\A}{{\mathbb A}}
\newcommand{\C}{{\mathbb C}}
\newcommand{\G}{{\mathbb G}}
\renewcommand{\P}{{\mathbb P}}
\newcommand{\Q}{{\mathbb Q}}
\newcommand{\Z}{{\mathbb Z}}
\renewcommand{\1}{{\mathbf{1}}}
\newcommand{\MZ}{{\mathbf{MZ}}}

\newcommand{\CH}{{\rm CH}}
\newcommand{\surj}{\twoheadrightarrow}
\newcommand{\inj}{\hookrightarrow}
\newcommand{\Pic}{{\rm Pic}}
\newcommand{\Hom}{{\rm Hom}}
\newcommand{\Map}{{\rm Map}}
\newcommand{\Spec}{{\rm Spec \,}}
\newcommand{\sHom}{{\mathcal{H}{om}}}
\newcommand{\holim}{\mathop{{\rm holim}}}
\newcommand{\Spt}{{\mathbf{Spt}}}
\newcommand{\Spc}{{\mathbf{Spc}}}
\newcommand{\Sm}{{\mathbf{Sm}}}
\newcommand{\SH}{{\mathbf{SH}}}

\renewcommand{\H}{{\mathbf{H}}}
\newcommand{\colim}{\mathop{\text{colim}}}
\newcommand{\wh}{\widehat}

\title{{\bf $\A^1$-contractibility of Koras--Russell threefolds}}
\author{Marc Hoyois, Amalendu Krishna, Paul Arne {\O}stv{\ae}r}
\date{}
\begin{document}
\maketitle
\begin{abstract}
Finite suspensions of Koras--Russell threefolds are shown to be contractible in $\A^1$-homotopy theory. 
\end{abstract}

\section{Introduction}\label{sec:Intro}
\indent
The cancellation problem in affine algebraic geometry 
asks whether the affine space $\A^{n}_{\C}$ is cancellative, i.e.,  
is any smooth complex affine variety $X$ with the property 
$X\times\A^1_{\C}\simeq\A^{n+1}_{\C}$ isomorphic to $\A^{n}_{\C}$? 
This problem is often referred to 
as the {\sl Zariski cancellation problem},
based on its birational version which was raised by Zariski in 1949
(see \cite[\S5]{Na}).

This has been settled affirmatively for curves by Abhyankar, Eakin, and Heinzer \cite{AEH},  
and for surfaces by Fujita \cite{Fujita}, Miyanishi and Sugie \cite{MS}. 
Koras--Russell threefolds are smooth complex contractible affine threefolds with a hyperbolic action of a one-dimensional torus with a unique fixed point, 
such that the quotients of the threefold and the tangent space of the fixed point by this action are isomorphic \cite{KR1}
(see \S\ref{sect:EKT} for a precise definition).
A vivid example is the Russell cubic hypersurface given by the equation $x+x^2y+z^2+t^3=0$ in $\A^4_{\C}=\Spec(\C[x,y,z,t])$ equipped with the $\C^{\times}$-action 
$\lambda\cdot (x,y,z,t)=(\lambda^6 x,\lambda^{-6}y,\lambda^3 z,\lambda^2 t)$. 

The Koras--Russell threefolds enjoy remarkable geometric properties exploited in the solution \cite{KKMLR} of the linearization problem for $\C^{\times}$-actions,
and have for long been believed to be counterexamples to the cancellation problem.
Gupta \cite{Gupta} showed that Asanuma threefolds \cite{Asanuma} give 
counterexamples to the cancellation problem for algebraically 
closed fields of positive characteristics. However, Gupta's 
counterexamples do not include the Koras--Russell threefolds.

In \cite{Limanov}, 
Makar-Limanov used his eponymous invariant to distinguish $\A^3_{\C}$ from the Russell cubic.
Dubouloz \cite{Dubouloz}, on the other hand, showed that the same invariant is unable to distinguish between $X\times\A^1_{\C}$ and $\A^{4}_{\C}$.
The closely related Makar-Limanov and Derksen invariants (see, e.g., \cite[\S7]{Zaidenberg}) are the only known tools to distinguish smooth complex 
contractible affine varieties of dimension $\geq 3$ from affine spaces, 
and such varieties are topologically indistinguishable from affine spaces because their underlying analytic spaces are diffeomorphic to affine spaces, 
as shown independently by Dimca and Ramanujam (see \cite[\S3]{Zaidenberg}).
We refer the reader to \cite{Kraft}, \cite{Kraft-1} for further background and history on
the cancellation problem.

An ambitious program for showing that $X$ is not $\A^1$-contractible and hence 
not stably isomorphic to an affine space was launched by Asok \cite{Asok}.
He transformed the problem into the realm of $\A^1$-homotopy theory \cite{MV} by asking for $\A^1$-homotopic obstructions to such a stable isomorphism. 
One motivation for Asok's program comes from $G$-equivariant homotopy theory, 
for $G$ a compact group, 
where equivariant topological $K$-theory is representable.
The corresponding equivariant algebro-geometric results have just recently been established in a joint work with Heller \cite{HKO} 
(see also the antecedent work of Deligne and Voevodsky \cite{Deligne}).

Asok's program for showing that $X$ is not stably isomorphic to $\A^4_{\C}$ has received attention from several workers in the field.
It consists of the following steps:
\begin{enumerate}
\item[(i)] 
The roots of unity $\mu_p \subsetneq \C^{\times}$ acts on $X$ with fixed points $X^{\mu_p} = X^{\C^{\times}}$ for all but finitely many primes $p$.
\item[(ii)]
If $f: X \to Y$ is a $\mu_p$-equivariant map of smooth complex varieties with $\mu_p$-action such that $X \to Y$ as well as $X^{\mu_p} \to Y^{\mu_p}$ are 
$\A^1$-weak equivalences (in the sense of \cite{MV}), 
then $f$ induces an isomorphism of $\mu_p$-equivariant Grothendieck groups.
\item[(iii)]
The group $K_0^{\mu_p}(X)$ is nontrivial for all but finitely many primes $p$, i.e., $K_0^{\mu_p}(\Spec\C)\to K_0^{\mu_p}(X)$ is not an isomorphism. 
\end{enumerate}
The third step was motivated by Bell's computation in \cite{Bell} which shows that $K_0^{\C^\times}(X)$ is nontrivial.
It follows from (i)-(iii) that $X$ is not $\A^1$-contractible and hence not stably isomorphic to $\A^3_{\C}$ (see \S\ref{sect:EKT}).  
This would be the first example of a smooth affine $\A^1$-connected variety that is topologically contractible, 
but not $\A^1$-contractible.
In fact, the only known $n$-dimensional complex affine 
$\A^1$-contractible variety is $\A^n_{\C}$ itself.
The results of this paper are motivated by the quest of verifying all the steps of Asok's program and its application to the cancellation problem.

In \propref{prop:fixed-point}, we use representation theory to show that (i) holds. 
Step (ii) is known to be false integrally \cite{Herrmann}.
However, 
one of the main results of \cite{HKO} shows representability of rationalized equivariant $K$-theory in the fixed point Nisnevich topology. 
It follows that step (ii) holds rationally.
Asok's program will therefore go through provided (iii) holds rationally.

In Theorem \ref{thm:KR-trivial} we show, 
using elementary arguments, 
that the $\mu_p$-equivariant Grothendieck group of a Koras--Russell threefold is trivial for almost all primes $p$.
Thus (iii) cannot be true even with integral coefficients, and the above program for showing that $X$ is not $\A^1$-contractible, 
and hence not stably isomorphic to an affine space, 
cannot succeed.

The failure of (iii) strengthens the potential of Koras--Russell threefolds as counterexamples to the cancellation problem. 
From a $K$-theoretic viewpoint, 
our computation provides strong evidence for $X$ being $\A^1$-contractible.
In \thmref{thm:Main-KR}, 
we show that the Russell cubic, 
and more generally a large class of Koras--Russell threefolds,
becomes $\A^1$-contractible after some finite suspension with the pointed projective line $(\C\P^1,\infty)$. 
This implies in particular that these threefolds have trivial Borel equivariant cohomology.
We generalize work of Murthy \cite[Corollary~3.8]{Murthy} by showing that all vector bundles on such threefolds are trivial. 
Koras and Russell posed the question of triviality of such vector bundles in \cite[\S17.2]{KR1}. 
In the case of the Russell cubic $X$, 
\thmref{thm:MC-Vanish} states that for any smooth complex affine variety $Y$, 
the projection map $X \times Y \to Y$ induces an isomorphism of (bigraded) integral motivic cohomology rings 
\begin{equation*}
H^{*, *}(Y ,\Z)\overset{\simeq}{\to}H^{*, *}(X \times Y,\Z).
\end{equation*}
This is a very special property of $X$, 
making it look like an affine space in the eyes of motivic cohomology.
By combining the above with the slice filtration technology and results of Levine \cite{Levine} and Voevodsky \cite{VoevSlice}, we deduce that the 
motivic suspension spectrum of $X$ is $\A^1$-contractible.
We note that $X$ is the first example of a smooth complex affine variety which is not isomorphic to an affine space in spite of the fact that it is stably $\A^1$-contractible.

\section{Equivariant $K_{0}$ of Koras--Russell threefolds}
\label{sect:EKT}

We begin by briefly explaining how Asok's program (i)-(iii) in the introduction implies that a Koras--Russell threefold is not $\A^1$-contractible, 
and hence not stably isomorphic to an affine space. Since $X$ has a unique $\C^\times$-fixed point, (i) and (iii) imply that there exists a prime $p$ such that
$X^{\mu_p}$ is $\A^1$-contractible (being just a point) and $K^{\mu_p}_0(X)$ is nontrivial. If $X$ itself were $\A^1$-contractible, step (ii) would thus contradict the nontriviality of $K^{\mu_p}_0(X)$. 

It follows that $X$ cannot be stably isomorphic to an affine space by the implication (where the symbol $\sim$ denotes $\A^1$-homotopy equivalence) 
\[
X \times \A^m_{\C}
\simeq 
\A^{m+3} _{\C}
\Rightarrow 
X \sim X \times \A^m_{\C} 
\sim
\A^{m+3}_{\C} 
\sim 
\Spec\C.
\]

Next we verify step (i). 
See \cite[Lemma~1.9]{Brion-Lum} and \cite[Chapter~2.5, Theorem~2]{Kraft} for proofs of the following result.
\begin{lem}\label{lem:embedding}
Let $G$ be a linear algebraic group acting on a complex affine variety $X$.
Then there exists a rational representation $V$ of $G$ and a $G$-equivariant closed embedding $X \inj \Spec({\rm Sym}(V^*))$.
\end{lem} 

\begin{prop}\label{prop:fixed-point} 
Let $\C^{\times}$ act on a complex affine variety $X$. 
Then there exists a finite set of positive integers $S$ such that  $X^{\mu_n} = X^{\C^{\times}}$ whenever $n>0$ is an integer relatively prime to all elements of $S$.
\end{prop}
\begin{proof}
By Lemma~\ref{lem:embedding}, we can find a $\C^{\times}$-equivariant closed embedding $i: X \inj V$, 
where $\C^{\times}$ acts linearly on $V = \A^r_\C$ with weights $(a_1, \cdots , a_r)$. 
Since $X^{H} = V^H \cap X$ for every closed subgroup $H \subseteq \C^{\times}$, we may assume $X = V$.
Now choose a $\C^{\times}$-equivariant decomposition $V = V_1 \times V_2$,
where $V_1$ and $V_2$ are $\C^{\times}$-invariant linear subspaces of $V$ such that $(V_1)^{\C^{\times}} = V_1$ and $(V_2)^{\C^{\times}} = \{0\}$.
Set $S=\{a_1,\cdots,a_r\}\setminus\{0\}$ and choose $n>0$ as in the formulation of the proposition. 
We finish the proof by showing $V^{\mu_n} \subseteq V^{\C^{\times}}$:
Let $(x_1,\cdots,x_m)$ and $(x_{m+1},\cdots,x_r)$ be points in $V_1$ and $V_2$, 
respectively.
Suppose $\underline{x}=(x_1,\cdots,x_m,x_{m+1},\cdots,x_r)\in V$ is not fixed by $\C^{\times}$ so that $x_i \neq 0$ for some $m+1\le i\le r$. 
If $\lambda(\underline{x})=\underline{x}$ for $\lambda \in \mu_n$ then $\lambda^{a_i}x_i = x_i$. 
Now since $x_i \neq 0$, 
$\lambda^{a_i} = 1$. 
Since $\lambda^n = 1$ and $n$ is relative prime to $a_i$, 
it follows that $\lambda = 1$.
In other words, $\underline{x} \notin V^{\mu_n}$.
\end{proof}

Our computation of $\mu_p$-equivariant Grothendieck groups of Koras--Russell threefolds relies on Proposition~\ref{prop:KR-finite} and a few elementary lemmas, 
starting with \cite[Theorem 2]{GD}.
\begin{lem}
\label{lem:NT}
Let $\Phi_n(t) \in \Z[t]$ be the cyclotomic polynomial whose roots are the primitive $n$th roots of unity.
Then the following hold for integers $0<m<n$.
\begin{enumerate}
\item
$\frac{\Z[t]}{\left(\Phi_m (t), \Phi_n (t)\right)} = 0$ if $\frac{n}{m}$ is not a prime power.
\item
$\frac{\Z[t]}{\left(\Phi_m (t), \Phi_n (t)\right)} = {\Z}/{q}$ if $\frac{n}{m}=q^i$ for some prime $q$ and $i\ge 1$.
\end{enumerate}
\end{lem}

\begin{lem}
\label{lem:NT-3}
Let $p$ be a prime and set $g(t) = \stackrel{r}{\underset{i=1}\prod} \Phi_{a_i}(t)$ for $\Phi_{a_i}(t)$ cyclotomic polynomials (not necessarily distinct) 
such that $a_i \ge 2$ and $a_i \notin (p)$ for $1 \le i \le r$. 
Then $\frac{\Z[t]}{(1-t^{p^n}, g(t))}$ is a finite ring for every $n \ge 1$.
\end{lem}
\begin{proof}
Since $g(t)$ is a monic polynomial (being a product of cyclotomic polynomials),
it follows that $\frac{\Z[t]}{(g(t))}$ is integral over $\Z$ and hence is finite over $\Z$. 
In particular, 
it is a finitely generated abelian group.
We conclude that $A:=\frac{\Z[t]}{(1-t^{p^n}, g(t))}$ is also a finitely generated abelian group. 
Using the structure theorem for such groups, it suffices to show that $A$ is a torsion group. 
Equivalently, we need to show that $\frac{\Q[t]}{(1-t^{p^n}, g(t))} = 0$.

It suffices to show that $1-t^{p^n}$ and $g(t)$ are relatively prime in $\Q[t]$. 
Suppose on the contrary that they have a common irreducible factor $f(t) \in \Q[t]$. 
We can write $1-t^{p^n} = \stackrel{n}{\underset{j=0}\prod}\Phi_{p^j}(t)$.
Since the cyclotomic polynomials are irreducible over $\Q$,
we conclude that there must exist $1 \le i \le r$ and $1 \le j \le n$ such that $f(t) = \Phi_{a_i}(t) = \Phi_{p^j}(t)$. 
This contradicts our assumption and finishes the proof.
\end{proof}

For a complex variety $X$ with action of an algebraic group $G$, 
let $K^G_0(X)$ denote the Grothendieck group of equivariant vector bundles on $X$ and let $R(G)\simeq K^G_0(\C)$ denote the representation ring of $G$. 
If $H \subseteq G$ is a closed subgroup, 
we note there is a restriction map $K^G_0(X) \to K^{H}_0(X)$.
 
\begin{prop}\label{prop:KR-finite} 
Let $X$ be a smooth affine variety with $\C^{\times}$-action and let $n>0$ be an integer. 
There is a natural ring isomorphism $\phi\colon K^{\C^{\times}}_0(X) {\underset{R(\C^{\times})}\otimes} R(\mu_n) \to K^{\mu_n}_0(X)$.
\end{prop}
\begin{proof}
The natural map $\iota: X \simeq X \stackrel{\mu_n}{\times} \mu_n \inj X \stackrel{\mu_n}{\times} \C^{\times} \simeq X \times ({\C^{\times}}/{\mu_n})$ induces 
an isomorphism $\iota^*: K^{\C^{\times}}_0(X \times ({\C^{\times}}/{\mu_n}))\xrightarrow{\simeq} K^{\mu_n}_0(X)$ \cite[Proposition~6.2]{Thomason1}. 
The exact sequence (in the {\'e}tale topology) of algebraic groups
\[
0 \to \mu_n \to \C^{\times} \xrightarrow{\psi} \C^{\times} \to 0
\]
identifies ${\C^{\times}}/{\mu_n}$ with $\C^{\times}$ which acts on itself by $a\star b=\psi(a)b=a^nb$. 
With this action, 
\begin{equation}\label{eqn:KR-finite-0}
\iota^*: 
K^{\C^{\times}}_0(X \times \C^{\times}) 
\xrightarrow{\simeq} 
K^{\mu_n}_0(X).
\end{equation}
Letting $\C^{\times}$ act on $\A^1$ with weight $n$, there is a localization exact sequence \cite[Theorems~2.7, 5.7]{Thomason1}
\begin{equation}\label{eqn:KR-finite-1}
K^{\C^{\times}}_0(X) \to K^{\C^{\times}}_0(X \times \A^1) \to K^{\C^{\times}}_0(X \times \C^{\times})
\to 0.
\end{equation}
By equivariant homotopy invariance \cite[Theorem~4.7]{Thomason1} and self-intersection \cite[Theorem~2.1]{VV}, 
the above exact sequence can be recast as
\begin{equation}\label{eqn:KR-finite-2}
K^{\C^{\times}}_0(X) 
\xrightarrow{1-t^n} K^{\C^{\times}}_0(X) 
\to 
K^{\C^{\times}}_0(X \times \C^{\times})
\to 
0,
\end{equation}
where we identify $R(\C^{\times})$ with $\Z[t, t^{-1}]$.
Combined with \eqref{eqn:KR-finite-0} we obtain the isomorphisms
\[
K^{\C^{\times}}_0(X) {\underset{R(\C^{\times})}\otimes} R(\mu_n) 
\xrightarrow{\simeq} 
K^{\C^{\times}}_0(X) {\underset{R(\C^{\times})}\otimes} \ \frac{R(\C^{\times})}{(1-t^n)}
\xrightarrow{\simeq}  K^{\C^{\times}}_0(X \times \C^{\times}) \xrightarrow{\simeq} 
K^{\mu_n}_0(X).
\]
This finishes the proof because $\phi$ is a ring map.
\end{proof}  

An algebraic action of $\G_m(\C) \simeq \C^{\times}$ on a smooth complex affine variety is called hyperbolic if it has a unique fixed point 
and the weights of the induced linear action on the tangent space at this fixed point are all non-zero and their product is negative.
Recall from \cite{KR1} that a {\sl Koras--Russell} threefold $X$ is a smooth hypersurface in $\A^4_\C$ which is
\begin{enumerate}
\item
topologically contractible,
\item
has a hyperbolic $\C^{\times}$-action, and
\item
the quotient $X//{\C^{\times}}$ is isomorphic to the quotient of the linear $\C^{\times}$-action on the tangent space at the fixed point (in the sense of GIT).
\end{enumerate}
It is shown in \cite[Theorem~4.1]{KR1} that the coordinate ring of a threefold $X$ satisfying (1)-(3) has the form  
\begin{equation}\label{eqn:KRCR}
\C[X] 
= 
\frac{\C[x,y, z,t]}{t^{\alpha_2} - G(x, y^{\alpha_1}, z^{\alpha_3})}. 
\end{equation}
Here $\alpha_1, \alpha_2, \alpha_3$ are pairwise relatively prime positive integers. 
Let $r$ denote the $x$-degree of the polynomial $G(x, y^{\alpha_1},0)$ and set $\epsilon_X = (r-1)(\alpha_2-1)(\alpha_3-1)$. 
A Koras--Russell threefold $X$ is said to be {\sl nontrivial} if $\epsilon_X \neq 0$.  

Next we compute the $\mu_p$-equivariant Grothendieck groups of Koras--Russell threefolds.
Bell \cite{Bell} showed that the $\C^{\times}$-equivariant Grothendieck group 
of $X$ is of the form
\begin{equation}\label{eqn:Bell-*}
\begin{array}{lll}
K^{\C^{\times}}_0(X) & = & {R(\C^{\times})} \oplus
\left(\frac{R(\C^{\times})}{(f(t))}\right)^{\rho-1} \\
& = & \Z[t, t^{-1}] \oplus \left(\frac{\Z[t, t^{-1}]}{(f(t))}\right)^{\rho-1} \\
& = & \Z[t, t^{-1}] \oplus \Z^{(\alpha_2 -1)(\alpha_3-1)},
\end{array}
\end{equation}
where 
\begin{equation}\label{eqn:Bell-*-1}
f(t) = \frac{(1-t^{\alpha_2 \alpha_3})(1-t)}{(1-t^{\alpha_2})(1-t^{\alpha_3})}
\end{equation}
is a polynomial of degree $(\alpha_2-1)(\alpha_3-1)$
and $\rho\ge 2$ is the number of irreducible factors of 
$G(x,y^{\alpha_1},0)\in \C[x,y]$. 
In particular, $K^{\C^{\times}}_0(X)$ is nontrivial.
We shall show that the $\mu_p$-equivariant Grothendieck group of $X$ is trivial for almost all primes $p$. 
This proves that non-$\A^1$-contractibility of a Koras--Russell threefold cannot be detected by $\mu_p$-equivariant Grothendieck groups, 
cf.~the discussion of step (ii) in Asok's program.
\begin{thm}
\label{thm:KR-trivial}
Let $p$ be a prime and let $n \ge 1$ be an integer.
Let $\mu_{p^n}$ act on a Koras--Russell threefold $X$ via the inclusion 
$\mu_{p^n} \subsetneq \C^{\times}$.
Let $K^{\mu_{p^n}}_0(X)$ denote the Grothendieck group of $\mu_{p^n}$-equivariant
vector bundles on $X$. 
Then the following hold.
\begin{enumerate}
\item
The structure map $X \to \Spec(\C)$ induces an isomorphism 
$R(\mu_{p^n}) \oplus F_{p^n} \xrightarrow{\simeq} K^{\mu_{p^n}}_0(X)$.
\item
$F_{p^n}$ is a finite abelian group which is nontrivial if and only if $X$ is 
nontrivial and $p|\alpha_2\alpha_3$. 
\end{enumerate}
\end{thm}
\begin{proof}
When $\epsilon_X=0$, it is known by \cite{KML} that $X \simeq \A^3$ with a linear $\C^{\times}$-action. 
Hence homotopy invariance of equivariant Grothendieck groups (see \cite[Theorem~4.7]{Thomason1}) shows that $R(\mu_{p^n})\xrightarrow{\simeq} K^{\mu_{p^n}}_0(X)$.
Thus we may assume $X$ is nontrivial so that $\alpha_2, \alpha_3\ge 2$. 

Recall that $R(\C^{\times})\simeq\Z[t,t^{-1}]$, 
where $t$ denotes the $\C^{\times}$-action on $\A^1$ by scalar multiplication. 
The monic polynomial $f(t)$ (see ~\eqref{eqn:Bell-*-1}) 
is related to cyclotomic polynomials by the equality
$$
f(t) = {\underset{d|\alpha_2 \alpha_3, d \nmid \alpha_2, d \nmid \alpha_3}\prod} \ \Phi_d(t)
$$ 
since 
$1-t^n={\underset{d|n}\prod}\ \Phi_d(t)$ and $(\alpha_2,\alpha_3)=1$.
Thus we can write 
\begin{equation}
\label{eqn:KR-trivial-2}
f(t)  
=  
{\underset{a | \alpha_2}{\underset{a \ge 2} \prod}} \ {\underset{b | \alpha_3}{\underset{b \ge 2}\prod}} \ \Phi_{ab}(t).
\end{equation}

Let $f_{p^n}(t)$ denote the image of $f(t)$ in $R(\mu_{p^n})$. 
Proposition \ref{prop:KR-finite} and ~\eqref{eqn:Bell-*} imply the isomorphism
\begin{equation*}
\label{eqn:KR-trivial-1}
K^{\mu_{p^n}}_0(X) 
\xrightarrow{\simeq} 
R(\mu_{p^n}) \oplus {\left(\frac{R(\mu_{p^n})}{(f_{p^n}(t))}\right)}^{\rho -1}.
\end{equation*}
Let $F_{p^n}$ be the shorthand for 
$$
{\left(\frac{R(\mu_{p^n})}{(f_{p^n}(t))}\right)}^{\rho-1}=
{\left(\frac{\Z[t]}{(f(t), 1-t^{p^n})}\right)}^{\rho-1}.
$$

Suppose that $p|\alpha_2\alpha_3$ so that $p$ divides either $\alpha_2$ or $\alpha_3$, but not both.
If $p|\alpha_2$, 
let $q \neq p$ be a prime divisor of $\alpha_3$.
Finiteness of $F_{p^n}$ follows from \eqref{eqn:KR-trivial-2} and Lemma \ref{lem:NT-3}.
Moreover,  
there are surjections $F_{p^n} \surj F_p \surj \frac{\Z[t]}{(\Phi_{p}(t),\Phi_{pq}(t))}\simeq {\Z}/{q}$, 
by Lemma~\ref{lem:NT}.
If $p  | \alpha_3$ then the same argument goes through.
We conclude that $F_{p^n}$ is finite and nontrivial if $p | \alpha_2 \alpha_3$.

Now suppose that $p$ does not divide $\alpha_2 \alpha_3$. 
In this case we show 
\begin{equation}
\label{eqn:f1tp}
\Z[t] = \left(f(t),1-t^{p^n}\right).
\end{equation}
It suffices to show \eqref{eqn:f1tp} for each irreducible factor of $f(t)$, 
and of $1-t^{p^n}$.
We are done by \eqref{eqn:KR-trivial-2} and Lemma \ref{lem:NT} since 
$a, b \notin (p)$ whenever $a, b \ge 2$ are such that $a | \alpha_2$ and
$b | \alpha_3$.
\end{proof}

By \cite[Proposition~2.5]{BH}, every exact sequence of 
$\mu_{p^n}$-equivariant vector bundles on $X$ splits as a direct sum of 
$\mu_{p^n}$-equivariant sheaves. 
When combined with \thmref{thm:KR-trivial} we obtain
\begin{cor}\label{cor:KR-trivial-tr}
Let $p$ be a prime number and $n \ge 1$ an integer.
Suppose $\mu_{p^n}$ acts on a Koras--Russell threefold $X$ via the inclusion 
$\mu_{p^n} \subsetneq \C^{\times}$ and $(p,\alpha_2\alpha_3)=1$.
Then every $\mu_{p^n}$-equivariant vector bundle on $X$ is stably trivial.
That is, 
for any $\mu_{p^n}$-equivariant vector bundle $E$ on $X$, there exist 
$\mu_{p^n}$-representations $F_1$ and $F_2$ such that $E \oplus F_1\simeq F_2$.
\end{cor} 

Borel equivariant $K$-theory can be defined in the context of unstable 
$\A^1$-homotopy theory by taking $K$-theory of the motivic Borel construction
(see \cite[\S~3.3]{Krishna}).
By \cite[Theorem 1.3]{Krishna}, the Borel $\C^\times$-equivariant Grothendieck group of $X$ is the completion of $K_0^{\C^\times}(X)$ at the augmentation ideal of $R(\C^\times)$. 
The next result shows that this completion is trivial.
Later we will be able to prove that, in fact, any Borel equivariant cohomology theory vanishes on $X$ (see Corollary~\ref{cor:Trivial-EC}).
Thus, Borel equivariant cohomological invariants cannot distinguish
a Koras--Russell threefold from an $\A^1$-contractible smooth affine variety.

\begin{thm}\label{thm:Complete}
Let $X$ be a Koras--Russell threefold.
Denote by $\wh{K^{\C^{\times}}_0(X)}$ the $I_{\C^{\times}}$-adic completion of the $R(\C^{\times})$-module $K^{\C^{\times}}_0(X)$.
Then there are ring isomorphisms 
\[
\Z[[t]] 
\xrightarrow{\simeq} 
\wh{R(\C^{\times})} 
\xrightarrow{\simeq} 
\wh{K^{\C^{\times}}_0(X)}.
\]
\end{thm} 
\begin{proof}
Letting $A=\Z[t, t^{-1}] \simeq R(\C^{\times})$ and $\wh{A} = \wh{R(\C^{\times})}$ there is an isomorphism $A/{(1-t)^n} \simeq {\Z[t]}/{(1-t)^n}$ for all $n \ge 0$. 
Using the automorphism $\Z[t]\to\Z[t]$ sending $t$ to $(1-t)$ we deduce the ring isomorphism ${\Z[t]}/{(t^n)} \to A/{(1-t)^n}$. 
We conclude that $\wh{A} \simeq {\underset{n}\varprojlim} \ {\Z[t]}/{(t^n)} \simeq \Z[[t]]$. 

Using \eqref{eqn:Bell-*}, it remains to show $\wh{A/{f(t)}} = 0$.
Since $\wh{A/{f(t)}} \simeq {\wh{A}}/{f(t)}$, we need to show that $f(t)$ is invertible in $\wh{A}\simeq \Z[[t]]$. 
We claim that $f(1) = \pm{1}\in \wh{A}$ or equivalently $(f(t),1-t)=A$. 
Using \eqref{eqn:KR-trivial-2} it suffices to show that $(\Phi_{ab}(t),1-t)=A$ for $a, b \ge 2$ and $(a,b)=1$.
This follows immediately from Lemma \ref{lem:NT} since $\Phi_1(t)=1-t$. 
\end{proof}

We end this section with an interesting application of Theorem \ref{thm:KR-trivial}.
To put this in context, 
we first recall a result due to Jackowski \cite[Theorem~4.4]{Jack} which can be viewed as a local-global principle for equivariant topological $K$-theory. 
\begin{thm}
\label{thm:Jack}
Let $G$ be a compact Lie group acting on a finite $CW$-complex $X$. 
Let $\sC(G)$ denote the set of finite cyclic subgroups of $G$. 
Then the restriction map
\[
K^G_*(X)_{\Q} 
\to 
{\underset{H \in \sC(G)}\prod} K^H_*(X)_{\Q}
\]
is a monomorphism.
\end{thm}

It is not known if an analogous result is true for equivariant algebraic $K$-theory. 
But we show that a weaker version of Theorem \ref{thm:Jack} fails for equivariant algebraic $K$-theory.  
Let $\sC_{\rm ind}(G)$ denote the set of indecomposable finite cyclic subgroups of $G$ (i.e., subgroups which cannot be written as a direct sum of two or more cyclic subgroups). 
\begin{thm}\label{thm:C-example}
For $X$ a nontrivial Koras--Russell threefold, the restriction map
\[
K^{\C^{\times}}_0(X)_{\Q} 
\to 
{\underset{H \in \sC_{\rm ind}(\C^{\times})}\prod} K^{H}_0(X)_{\Q}
\]
is not a monomorphism.
\end{thm}
\begin{proof}
The rationalized version of 
~\eqref{eqn:Bell-*}, i.e., the isomorphism 
\begin{equation*}
\label{eqn:KR-trivial-0-1*}
K^{\C^{\times}}_0(X)_{\Q} 
\xrightarrow{\simeq} 
\Q[t^{\pm 1}] 
\oplus
{\left(\frac{\Q[t^{\pm 1}]}{(f(t))}\right)}^{\rho -1}
\end{equation*}
implies in combination with Proposition \ref{prop:KR-finite}, the isomorphism
$$
K^{\mu_n}_0(X)_{\Q} 
\xrightarrow{\simeq} 
\Q[t^{\pm 1}] 
\oplus
{\left(\frac{\Q[t^{\pm 1}]}{(f(t),1-t^n)}\right)}^{\rho -1}.
$$

An indecomposable finite cyclic subgroup $H$ of $\C^{\times}$ is of the form ${\Z}/{p^n}$ for some prime $p$, 
and integer $n \ge 0$. 
Under the restriction map $K^{\C^{\times}}_0(X)_{\Q} \to K^{\mu_{p^n}}_0(X)_{\Q}$, 
$$ 
F_{\C^{\times}} 
= 
{\left(\frac{\Q[t^{\pm 1}]}{(f(t))}\right)}^{\rho -1}
\to
F_{p^n} 
= 
{\left(\frac{\Q[t^{\pm 1}]}{(f(t), 1-t^{p^n})}\right)}^{\rho -1}.
$$
Theorem \ref{thm:KR-trivial} implies $F_{p^n} = 0$.
On the other hand, 
Bell \cite[Theorem 5.3]{Bell} has shown that $F_{\C^{\times}}$ is a nonzero finite dimensional $\Q$-vector space.
\end{proof}

\section{Motivic cohomology of Koras--Russell threefolds}
\label{sect:MCKRT}
Letting $\C[X]$ denote the coordinate ring of a complex variety $X$, 
the Makar-Limanov invariant is defined as the subring of functions invariant under all possible $\C^{+}$-actions on $X$,
i.e., 
\[
ML(X) = {\underset{\C^{+} \to {\rm Aut}(X)}\bigcap} {\C[X]}^{\C^{+}}.
\]
We note that $ML(\A^n_\C)=\C$ for all $n \ge 1$.
Kaliman and Makar-Limanov \cite{KML} showed that $ML(X)\not\simeq\C$ for a nontrivial Koras--Russell threefold. 
In fact, 
their main result shows $ML(X) = \C[X]$ unless $X$ is isomorphic to a hypersurface in $\A^4_\C$ given by one of the following equations:
\begin{equation}\label{eqn:KRFK}
ax + x^my + z^{\alpha_2} + t^{\alpha_3}
\end{equation}
\begin{equation}\label{eqn:KRSK}
ax + (x^b+ z^{\alpha_2})^ly + t^{\alpha_3}.
\end{equation}
Here, $a \in {\C}^{\times}$ and $l, m, b, \alpha_2,  \alpha_3 \ge 2$ are integers such that $(\alpha_2, \alpha_3) =1$ in \eqref{eqn:KRFK} 
and $(\alpha_2, b \alpha_3) =1$ in \eqref{eqn:KRSK}.
We shall refer to threefolds given by \eqref{eqn:KRFK} as {\sl Koras--Russell threefolds of the first kind}, 
and by \eqref{eqn:KRSK} as {\sl Koras--Russell threefolds of the second kind}.
The Russell cubic is a Koras--Russell threefold of the first kind.
By \cite{KML}, 
$ML(X)={\rm Image}\left(\C[x] \to \C[X]\right)$ if $X$ is of the first kind, 
and $ML(X)={\rm Image}\left(\C[x,z] \to \C[X]\right)$ if $X$ is of the second kind.
\vspace{0.1in}

In \cite{Bloch}, 
Bloch defined higher Chow groups $\CH_j(X,i)$ as a candidate for motivic cohomology, 
i.e., 
an algebraic singular (co)homology theory.
Set
\begin{equation*}
\CH_*(X, i) = {\underset{j \ge 0}\bigoplus} \CH_j(X,i) 
\text{ and }
\CH_{\star}(X) = {\underset{i \ge 0}\bigoplus} \CH_*(X,i).
\end{equation*}
If $X$ is smooth of dimension $d$ we write $\CH^j(X,i) = \CH_{d-j}(X,i)$.
Bloch has shown that the higher Chow groups are contravariantly functorial for flat maps and covariantly functorial for proper maps \cite[Proposition 1.3]{Bloch}.
The important property of homotopy invariance is shown in \cite[Theorem 2.1]{Bloch}.
Recall also that Bloch's higher Chow groups are invariant under 
nilpotent extensions, i.e., $\CH_j(X_{\rm red},i) \simeq \CH_j(X,i)$
for any quasi-projective $\C$-scheme $X$. This is an elementary
consequence of the fundamental localization result in \cite{Bloch2}.
We shall make repeated use of this nil-invariance 
and the following result (see \cite[Corollary~22.6]{Matsumura}).
\begin{lem}
\label{lem:M-flat}
Let $A$ be a Noetherian ring, 
$B = A[x_1, \cdots , x_n]$ the polynomial ring over $A$, 
and let $f\in B$. 
If the ideal generated by the coefficients of $f$ contains $1$, 
then $f$ is a non-zero-divisor in $B$, 
and $B/{(f)}$ is flat over $A$.
\end{lem}

\begin{prop}\label{prop:Contr-0}
Let $a, b \ge 2$ be relatively prime integers and $Y$ a smooth complex affine variety.
The natural ring extension $\C \to {\C[u, v]}/{(u^a + v^b)}$ induces an isomorphism of higher Chow groups 
\begin{equation*}
\CH_{\star}(\C[Y]) 
\overset{\simeq}{\to}
\CH_{\star}\left({\C[Y][u, v]}/{(u^a + v^b)}\right).
\end{equation*}
\end{prop} 
\begin{proof}
Set $A={\C[u, v]}/{(u^a + v^b)}$.
Since $(a, b) = 1$, we may assume that $b$ is an odd number. 
The ring homomorphism $\phi:A \to \C[t]$ given by $\phi(u)=t^b$, $\phi(v)=-t^a$ is injective with image the subring $\C[t^a, t^b] \inj \C[t]$. 
In particular, 
$A$ is an integral domain.
The ring map $\psi: \C[u] \to A$ is an injection of a PID into an integral domain, 
hence it is flat.
We have a commutative diagram:
\begin{equation}
\label{eqn:Contr-0-1}
\xymatrix@C.9pc{
\C \ar[r] \ar[d]_{=} & \C[u] \ar[d] \ar[r]^-{\psi} & A \ar[d] \ar[r]^-{\phi} &
\C[t] \ar[d] \\
\C \ar[r] & \C[u^{\pm 1}] \ar[r] & A[u^{\pm 1}] \ar[r]^-{\simeq} &
\C[t^{\pm 1}]}
\end{equation}
As indicated in \eqref{eqn:Contr-0-1},
inverting $u$ turns the normalization map $\phi$ into an isomorphism, 
and the composite $\C[u^{\pm1}] \to \C[t^{\pm 1}]$ is determined by $\psi(u)=t^b$. 
Note that we may replace $\C$ by $B=\C[Y]$ in the discussion leading to \eqref{eqn:Contr-0-1}.

For $i \ge 0$ there is an induced commutative diagram of higher Chow groups:
\begin{equation}\label{eqn:Contr-0-2}
\xymatrix@C.8pc{ 
& \CH_*(B, i) \ar[d]_{\simeq}^{\alpha_1} \ar@/_6.5pc/[dddd]_{=} \ar[dr]^{\alpha_2} 
& & \\ 
0 \ar[r] & \CH_*(B[u], i) \ar[d]^{\psi^{\ast}_1} \ar[r]^<<{j^{\ast}_1} &
\CH_*(B[u^{\pm 1}], i) \ar[d]^{\psi^{\ast}_2} \ar[r]^{\partial_1}  &  \CH_*(B, i-1) 
\ar[d]^{\psi^{\ast}_3} \ar[r] &  0 \\
0 \ar[r] &  \CH_*(A{\underset{\C}\otimes} B, i) \ar[r]^<<<{j^{\ast}_2} & 
\CH_*(A{\underset{\C}\otimes} B[u^{\pm 1}], i) 
\ar[r]^<<<{\partial_2} 
&  \CH_*({A{\underset{\C}\otimes} B}/{(u)}, i-1) \ar[r] &  0 \\
0 \ar[r] & \CH_*(B[t], i) \ar[u]_{\phi_{\ast}^1} \ar[r]^<<<{j^{\ast}_3} &
\CH_*(B[t^{\pm 1}], i) \ar[u]_{\phi_{\ast}^2} \ar[r]^{\partial_3} &  \CH_*(B, i-1) 
\ar[u]_{\phi_{\ast}^3} \ar[r] & 0 \\
& \CH_*(B, i) \ar[u]^{\simeq}_{\beta_1} \ar[ur]_{\beta_2}. & & }
\end{equation}
The top vertical map, 
induced by $\C \inj \C[u]$, 
is an isomorphism by homotopy invariance.
The vertical maps in the upper square are flat pullbacks and likewise for the lower square and proper pushforwards. 
It is easy to check the top and bottom rows are exact.  
Now $\phi^2_{\ast}$ and $\phi^3_{\ast}$ are induced by a ring isomorphism and a nilpotent ring extension, 
respectively, 
so both of them are isomorphisms. 
Thus the middle row is also exact.

Since $\phi: A[u^{\pm 1}] \to \C[t^{\pm 1}]$ is an isomorphism of $\C$-algebras,
\eqref{eqn:Contr-0-1} and \eqref{eqn:Contr-0-2} imply $\psi^{\ast}_1$ is surjective:
its image 
$
{\rm Im}(\psi^{\ast}_1)
=
{\rm Im}(\psi^{\ast}_1 \circ \alpha_1)
=
{\rm Im}(j^{\ast}_2 \circ \psi^{\ast}_1 \circ \alpha_1)
=
{\rm Im}(\psi^{\ast}_2 \circ \alpha_2) 
=
{\rm Im}(\phi^2_{\ast} \circ \beta_2)
$
identifies with 
$
\phi^2_{\ast}\left({\rm Ker}(\phi^3_{\ast} \circ \partial_3)\right)
=
\phi^2_{\ast}\left({\rm Ker}(\partial_2 \circ \phi^2_{\ast})\right)
=
{\rm Ker}(\partial_2)
=
\CH_*(A{\underset{\C}\otimes} B,i)
$.
Since $\psi^{\ast}_2 \circ \alpha_2 = \phi^2_{\ast} \circ \beta_2$ is injective,
so is $j^{\ast}_2 \circ \psi^{\ast}_1 \circ \alpha_1 = \psi^{\ast}_2 \circ j^{\ast}_1 \circ \alpha_1$ and hence also $\psi^{\ast}_1$.
\end{proof}

\begin{prop}\label{prop:Contr-1}
Let $X$ be a Koras--Russell threefold of the first kind with coordinate ring
\begin{equation*}\label{eqn:Contr-1-1}
\C[X] = \frac{\C[x,y,z,t]}{\left(ax + x^my + z^{\alpha_2} + t^{\alpha_3}
\right)},
\end{equation*}
where $m>0$ is an integer, $a \in {\C}^{*}$, and $\alpha_2, \alpha_3 \ge 2$ are relatively prime integers.
For $Y$ any smooth complex affine variety, 
the pullback map $\CH^{\star}(Y) \to \CH^{\star}(X \times Y)$ is an isomorphism.
\end{prop}
\begin{proof}
Since $X$ and $Y$ are both smooth, 
it is equivalent to show that the flat pullback map $\CH_{\star}(Y) \to \CH_{\star}(X \times Y)$ is an isomorphism.
Setting $A=\C[Y][x]$ and $B=\C[X \times Y]$, the natural ring extension $\phi: A \to B$ is flat by Lemma \ref{lem:M-flat}. 
Now $\CH_{\star}(\C[Y]) \to \CH_{\star}(A)$ is an isomorphism by homotopy invariance, 
so it suffices to show $\phi^{\ast}:\CH_{\star}(A) \to \CH_{\star}(B)$ is an isomorphism. 

For $i\in\Z$, there is a commutative diagram of localization exact sequences, 
where the vertical pullback maps are induced by the flat map $\phi$:
\begin{equation}\label{eqn:Contr-1-2}
\xymatrix@C.5pc{
\CH_*(A[x^{-1}], i+1) \ar[r] \ar[d] & \CH_*(A/{(x)}, i) \ar[r] \ar[d] & 
\CH_*(A, i) \ar[r] \ar[d] & \CH_*(A[x^{-1}], i) \ar[r] \ar[d] & 
\CH_*(A/{(x)}, i-1) \ar[d] \\
\CH_*(B[x^{-1}], i+1) \ar[r] & \CH_*(B/{(x)}, i) \ar[r] & 
\CH_*(B, i) \ar[r] & \CH_*(B[x^{-1}], i) \ar[r] & 
\CH_*(B/{(x)}, i-1).}
\end{equation}

By identifying $A[x^{-1}] \to B[x^{-1}]$ with the ring extension $A[x^{-1}] \to A[x^{-1}][z,t]$, it follows that the leftmost and the second rightmost vertical 
maps in \eqref{eqn:Contr-1-2} are isomorphisms.

The map ${A}/{(x)} \to {B}/{(x)}$ coincides with the composite of the ring extensions
\begin{equation*}
\C[Y] \to \frac{\C[Y][z,t]}{\left(z^{\alpha_2} + t^{\alpha_3}\right)} 
\to
\left(\frac{\C[Y][z,t]}{\left(z^{\alpha_2} + t^{\alpha_3}\right)}\right)[y].
\end{equation*}
The first map induces an isomorphism of higher Chow groups by Proposition \ref{prop:Contr-0}, 
and likewise for the second map by homotopy invariance. 
It follows that the second vertical arrow from left and the rightmost vertical arrow 
in \eqref{eqn:Contr-1-2} are isomorphisms. 

We conclude using the 5-lemma that the middle vertical arrow in \eqref{eqn:Contr-1-2} is an isomorphism. 
\end{proof}

\begin{lem}\label{lem:Contr-2}
Let $a, b, c \ge 2$ be integers such that $(a,bc) =1$ and let $l>0$ be an integer.
Consider 
\begin{equation*}
\label{eqn:Contr-2-1}
A_l = \frac{\C[Y][x, z]}{\left((z^a + x^b)^l\right)} 
\text{ and the inclusion }
\phi: A_l \to \frac{A_l[t]}{(t^c-x)}.
\end{equation*}
Then $\phi$ is flat and the induced pullback map $\CH_{\star}(A_l) \to \CH_{\star}\left({A_l[t]}/{(t^c-x)}\right)$ is an isomorphism.
\end{lem}
\begin{proof}
Lemma \ref{lem:M-flat} shows $\phi$ is flat.
We may assume $l=1$ since nilpotent invariance implies the pushforward maps 
$$
\CH_{\star}(A_l) \to \CH_{\star}(A_1) 
\text{ and }
\CH_{\star}\left(\frac{A_l[t]}{(t^c-x)}\right) \to\CH_{\star}\left(\frac{A_1[t]}{(t^c-x)}\right) 
$$
are isomorphisms. 
Proposition \ref{prop:Contr-0} implies the ring extension $\C[Y] \inj A_1$ induces an isomorphism $\CH_{\star}(\C[Y]) \xrightarrow{\simeq} \CH_{\star}(A_1)$.
Next we prove there is an isomorphism 
$$
\CH_{\star}(\C[Y]) \to \CH_{\star}\left(\frac{A_1[t]}{(t^c-x)}\right).
$$
The ring extension $\C[Y] \to {A_1[t]}/{(t^c-x)}$ coincides with $\C[Y] \to {\C[Y][x, z, t]}/{(x- t^c, x^b + z^a)}$, 
and the latter ring can be identified with ${\C[Y][z,t]}/{(z^a + t^{bc})}$. 
Using the assumption $(a, bc) = 1$, Proposition \ref{prop:Contr-0} furnishes the desired isomorphism  
$$
\CH_{\star}(\C[Y])
\overset{\simeq}{\to} 
\CH_{\star}\left(\frac{A_1[t]}{(t^c-x)}\right)\simeq \CH_{\star}\left(\frac{\C[Y][z,t]}{(z^a + t^{bc})}\right).
$$ 
\end{proof}

\begin{prop}
\label{prop:Contr-3}
Let $X$ be a Koras--Russell threefold of the second kind with coordinate ring
\begin{equation*}\label{eqn:Contr-3-1}
\C[X] = \frac{\C[x,y,z,t]}{\left(ax + (x^b+ z^{\alpha_2})^ly + t^{\alpha_3}
\right)},
\end{equation*}
where $l>0$ is an integer, $a\in {\C}^{\times}$, and $b, \alpha_2, \alpha_3 \ge 2$ are integers such that $(\alpha_2,b\alpha_3)=1$.
For $Y$ any smooth complex affine variety, 
the pullback map $\CH^{\star}(\C[Y]) \to \CH^{\star}(X \times Y)$ is an isomorphism.
\end{prop}
\begin{proof}
Since $X$ and $Y$ are smooth, 
it is equivalent to show that the flat pullback map $\CH_{\star}(\C[Y]) \to \CH_{\star}(X \times Y)$ is an isomorphism.
We may assume $a=-1$. 
Setting $A = \C[Y][x, z]$ and $B = \C[X \times Y]$, the natural map $\phi: A \to B$ is flat by Lemma \ref{lem:M-flat}.
Moreover, 
$\CH_{\star}(\C[Y]) \to \CH_{\star}(A)$ is an isomorphism by homotopy invariance.
It remains to show $\phi^{\ast}:\CH_{\star}(A) \to \CH_{\star}(B)$ is an isomorphism. 

We set $u = (x^b+ z^{\alpha_2})^l \in A$ and consider for $i\in\Z$ the commutative diagram of exact localization sequences with 
vertical pullback maps induced by the flat map $\phi$:
\begin{equation}\label{eqn:Contr-3-2}
\xymatrix@C.5pc{
\CH_*(A[u^{-1}], i+1) \ar[r] \ar[d] & \CH_*(A/{(u)}, i) \ar[r] \ar[d] & 
\CH_*(A, i) \ar[r] \ar[d] & \CH_*(A[u^{-1}], i) \ar[r] \ar[d] & 
\CH_*(A/{(u)}, i-1) \ar[d] \\
\CH_*(B[u^{-1}], i+1) \ar[r] & \CH_*(B/{(u)}, i) \ar[r] & 
\CH_*(B, i) \ar[r] & \CH_*(B[u^{-1}], i) \ar[r] & 
\CH_*(B/{(u)}, i-1).}
\end{equation}

The  ring extension $A[u^{-1}] \to B[u^{-1}]$ coincides with $A[u^{-1}] \to A[u^{-1}][t]$. 
Hence the leftmost and the second rightmost vertical arrows in \eqref{eqn:Contr-3-2} are isomorphisms by homotopy invariance.

The ring extension ${A}/{(u)} \to {B}/{(u)}$ coincides with $A' \to \frac{A'[t,y]}{(t^{\alpha_3}-x)}$, 
where $A'= \frac{\C[Y][x, z]}{\left((z^{\alpha_2} + x^b)^l\right)}$.
There exist induced pullback maps
\[
\CH_{\star}(A') 
\to 
\CH_{\star}\left(\frac{A'[t]}{(t^{\alpha_3}-x)}\right) 
\to
\CH_{\star}\left(\frac{A'[t,y]}{(t^{\alpha_3}-x)}\right).
\]
Here, 
the first map is an isomorphism by Lemma \ref{lem:Contr-2} and the second by homotopy invariance.
Hence the second leftmost and the rightmost vertical maps in \eqref{eqn:Contr-3-2} are isomorphisms. 

Using the 5-lemma we conclude the middle vertical arrow \eqref{eqn:Contr-3-2} is an isomorphism.  
\end{proof}

Combining Propositions \ref{prop:Contr-1} and \ref{prop:Contr-3} with the isomorphism between higher Chow groups and motivic cohomology, 
as shown by Voevodsky \cite[Corollary~2]{Voev}, 
we obtain the following.
\begin{thm}\label{thm:MC-Vanish}
Let $X$ be a Koras--Russell threefold of the first or second kind, 
and let $Y$ be a smooth complex affine variety. 
Then the pullback map $H^{*, *}(Y,\Z) \to H^{*, *}(X \times Y,\Z)$ induced by the projection $X \times Y \to Y$
is an isomorphism of (bigraded) integral motivic cohomology rings.
\end{thm}

An immediate consequence of \thmref{thm:MC-Vanish} is that
$\CH^{\ge 1}(X) = 0$ for a Koras--Russell threefold of the first or second kind.
This result has the following application to the vector bundles on
such threefolds.

\begin{cor}\label{cor:trivial-bundle}
Let $X$ be a Koras--Russell threefold of the first or second kind. 
Then every vector bundle on $X$ is trivial.
\end{cor}
\begin{proof}
Serre showed that every vector bundle on $X$ of rank $\geq 4$ is a direct sum of a rank three vector bundle and a trivial bundle. 
Now $\CH^3(X)=0$,
see \thmref{thm:MC-Vanish}, 
so by \cite[Corollary~2.4]{KM} every vector bundle of rank three is a direct sum of a rank two vector bundle and the trivial line bundle.
Since $\CH^2(X) = 0$ by \thmref{thm:MC-Vanish}, 
it follows from \cite[Corollary~2]{AF} that every vector bundle of rank two is a direct sum of a line bundle and the trivial line bundle. 
Finally, 
as $\Pic(X) \simeq \CH^1(X) =0$,
every line bundle over $X$ is trivial. 
\end{proof}
\begin{remk}
\label{remk:Murthy}
Triviality of vector bundles on Koras--Russell threefolds of the first kind was shown by Murthy \cite[Corollary~3.8]{Murthy} by a completely different method. 
\end{remk}

\section{$\A^1$-contractibility of Koras--Russell threefolds}
\label{sect:STBCONT}
In this section, we prove our main results on $\A^1$-contractibility of Koras--Russell threefolds using \thmref{thm:MC-Vanish} and $\A^1$-homotopy theory 
(see, e.g., \cite{DLORV}, \cite{MV}).

The pointed unstable and stable $\A^1$-homotopy categories $\H_\bullet(\C)$ and $\SH(\C)$ are the homotopy categories of simplicial model categories 
$\Spc_\bullet(\C)$, respectively $\Spt(\C)$.
Recall that the objects of $\Spc_\bullet(\C)$ are pointed simplicial presheaves on $\Sm_\C$,
the Nisnevich site of smooth complex varieties of finite type, 
while $\Spt(\C)$ is comprised of $T$-spectra in $\Spc_\bullet(\C)$ for the Tate object $T=(\C\P^1,\infty)$.

We note that $\SH(\C)$ is a triangulated category with shift functor $E\mapsto E[1]$ given by smashing with the topological circle. 
Denote by $\Sigma^\infty_T (X,x)\in\SH(\C)$ the $(\C\P^1,\infty)$-suspension spectrum of $X\in\Sm_\C$ and a rational point $x\in X(\C)$.
For fixed $F\in\SH(\C)$, 
recall from \cite{RO} that $E\in\SH(\C)$ is called 
\begin{itemize}
\item \emph{$F$-acyclic} if $E\wedge F\simeq \ast$;
\item \emph{$F$-local} if $\Hom_{\SH(\C)}(D,E)=0$ for every $F$-acyclic spectrum $D$.
\end{itemize}
It is clear that the $F$-local spectra form a colocalizing subcategory of $\SH(\C)$. 
Note that if $F$ is a ring spectrum, 
then any $F$-module $E$ is $F$-local (every map $D\to E$ factors through $D\wedge F$ and hence it is trivial if $D$ is $F$-acyclic).

Let $\MZ\in\SH(\C)$ denote the ring spectrum that represents motivic cohomology,
i.e.,  
for every $X\in\Sm_\C$ and integers $n,i\in\Z$ there is an isomorphism 
\begin{equation}
\label{eqn:motivic-coh-SH}
H^{n,i}(X,\Z)\simeq 
\Hom_{\SH(\C)}(\Sigma^\infty_T X_+,\MZ(i)[n]).
\end{equation}
Here, 
for $E\in\SH(\C)$, 
the Tate twist $E(1)$ is defined by $E(1)=E\wedge\Sigma^\infty_T(\G_m,1)[-1]$.

\begin{lem}
\label{lem:HZ-local}
For every $X\in\Sm_\C$ and closed point $x\in X$, 
$\Sigma^\infty_T (X,x)\in\SH(\C)$ is $\MZ$-local.
\end{lem}
\begin{proof}
By resolution of singularities, 
$\Sigma^\infty_T(X,x)$ is in the smallest thick subcategory of $\SH(\C)$ containing $\Sigma^\infty_T Y_+$ for any smooth projective variety $Y$
\cite[Theorem 1.4]{Riou}, 
\cite[Theorem 52]{RO:mz}.
It suffices now to show that $\Sigma^\infty_T Y_+$ is $\MZ$-local for such $Y$.
Our proof employs Voevodsky's slice filtration on $\SH(\C)$ \cite{VoevOpen}.
Recall the slice filtration of $E\in\SH(\C)$ is a tower of spectra
\[
\dotsb \to f_{q+1}E \to f_qE \to f_{q-1}E \to \dotsb \to E,
\quad q\in\Z.
\]
The $q$th slice $s_qE$ of $E$ is defined by the distinguished triangle
\[
f_{q+1}E \to f_qE \to s_qE \to f_{q+1}E[1].
\]
By Levine \cite[Theorem~3]{Levine}, 
the slice filtration of $\Sigma^\infty_TY_+$  for $Y$  a smooth projective variety is complete, 
i.e.,
\[
\holim_{q\to\infty} f_q(\Sigma^\infty_T Y_+)
\simeq 
\ast.
\]
Equivalently, 
if we define $c_qE$ by the distinguished triangle $f_qE \to E \to c_qE \to f_qE[1]$,
then
\[
\Sigma^\infty_T Y_+\simeq \holim_{q\to \infty} c_q(\Sigma^\infty_T Y_+).
\]
Since the subcategory of $\MZ$-local spectra is colocalizing, 
it now suffices to prove that $c_q(\Sigma^\infty_T Y_+)$ is $\MZ$-local for every $q\in \Z$.
By definition of the slice filtration, we have $c_q(\Sigma^\infty_T Y_+)\simeq\ast$ for $q\le 0$. 
Using the distinguished triangles
\[
s_qE \to c_qE \to c_{q-1}E \to s_qE[1]
\]
and induction on $q$, 
we are reduced to proving the slices $s_q(\Sigma^\infty_T Y_+)$ are $\MZ$-local. 
In fact, 
all slices in $\SH(\C)$ are $\MZ$-local: 
by \cite[\S6 (iv),(v)]{GRSO}, \cite[Theorem 3.6.13(6)]{Pelaez}, \cite[\S3]{Spitzweck}, 
any slice $s_qE$ is a module over the zeroth slice $s_0(\1)$ of the sphere spectrum, 
and hence it is $s_0(\1)$-local.
Moreover, 
$s_0(\1)\simeq\MZ$, as shown in \cite[Theorem~6.6]{VoevSlice}.
\end{proof}

\begin{thm}
\label{thm:Main-KR}
Let $X$ be a Koras--Russell threefold of the first or second kind. 
Then there exists an integer $n \ge 0$ such that $\Sigma^n_T(X, 0)$ is $\A^1$-contractible.
\end{thm}
\begin{proof}
We first reformulate Theorem \ref{thm:MC-Vanish} as an equivalence in the stable $\A^1$-homotopy category $\SH(\C)$, 
using its structure of a closed symmetric monoidal category. 
Denote by $\sHom(E,F)$ the internal homomorphism objects of $\SH(\C)$ characterized by the adjunction isomorphism
\[
\Hom_{\SH(\C)}(D,\sHom(E,F))
\simeq 
\Hom_{\SH(\C)}(D\wedge E,F).
\]
The structure map $X\to \Spec(\C)$ induces a morphism in $\SH(\C)$
\[
\MZ
\simeq 
\sHom(\Sigma^\infty_T \Spec(\C)_+,\MZ)
\to
\sHom(\Sigma^\infty_TX_+,\MZ).
\]
In view of \eqref{eqn:motivic-coh-SH}, 
Theorem \ref{thm:MC-Vanish} asserts that for every smooth complex affine variety $Y$ and $n,i\in\Z$, 
there is an induced isomorphism
\[
\Hom_{\SH(\C)}(\Sigma^\infty_TY_+(i)[n],\MZ)
\to
\Hom_{\SH(\C)}(\Sigma^\infty_TY_+(i)[n], \sHom(\Sigma^\infty_T X_+,\MZ)).
\]
The objects $\Sigma^\infty_T Y_+(i)[n]$ form a family of generators of $\SH(\C)$,
because
every smooth variety admits an open covering by smooth affine varieties.
We deduce that $\MZ\to \sHom(\Sigma^\infty_TX_+,\MZ)$ is an isomorphism.
Its retraction $\sHom(\Sigma^\infty_T X_+,\MZ)\to\MZ$ induced by the base point $0\in X$ is therefore also an isomorphism.
From the distinguished triangle
\[
\MZ[-1]
\to
\sHom(\Sigma^\infty_T (X,0),\MZ)
\to
\sHom(\Sigma^\infty_T X_+,\MZ)
\to
\MZ, 
\]
we deduce that $\sHom(\Sigma^\infty_T (X,0),\MZ)\simeq\ast$. 
By \cite[Theorems 1.4 and 2.2]{Riou} or \cite[Theorem 52]{RO:mz}, 
$\Sigma^\infty_T (X,0)$ is strongly dualizable in $\SH(\C)$, 
so that
\[
\sHom(\Sigma^\infty_T (X,0),\MZ)
\simeq
\sHom(\Sigma^\infty_T (X,0),\1)\wedge\MZ.
\]
Thus $\sHom(\Sigma^\infty_T (X,0),\1)$ is $\MZ$-acyclic, 
and for any $E\in\SH(\C)$ we obtain
\[
\Hom_{\SH(\C)}(E,\Sigma^\infty_T(X,0)\wedge\MZ)
\simeq
\Hom_{\SH(\C)}(E\wedge \sHom(\Sigma^\infty_T (X,0),\1),\MZ)
\simeq
\ast,
\]
since $E\wedge \sHom(\Sigma^\infty_T(X,0),\1)$ is $\MZ$-acyclic and $\MZ$ is $\MZ$-local (being an $\MZ$-module). 
By the Yoneda lemma, this implies
\[
\Sigma^\infty_T(X,0)\wedge\MZ\simeq\ast,
\]
i.e., $\Sigma^\infty_T(X,0)$ is $\MZ$-acyclic. 
On the other hand, 
by Lemma \ref{lem:HZ-local}, 
$\Sigma^\infty_T(X,0)$ is $\MZ$-local.
It follows that every endomorphism of $\Sigma^\infty_T(X,0)$ is trivial, 
and hence $\Sigma^\infty_T(X,0)\simeq\ast$. 
The proof is completed by the following lemma.
\end{proof}

\begin{lem}
\label{lem:compactness}
Let $X$ be a smooth complex variety and $x\in X$ a closed point.
If $\Sigma^\infty_T(X,x)\simeq\ast$ in $\SH(\C)$, 
then there exists an integer $n\geq 0$ such that $\Sigma^n_T(X,x)$ is $\A^1$-contractible.
\end{lem}
\begin{proof}
By \cite[Definition 2.10]{DRO}, an object $F\in\Spc_\bullet(\C)$ is fibrant exactly when for every $X\in\Sm_\C$,
(1) $F(X)$ is a Kan complex; 
(2) the projection $X\times\A^1\to X$ induces a homotopy equivalence $F(X)\simeq F(X\times\A^1)$;
(3) $F$ maps Nisnevich elementary distinguished squares in $\Sm_\C$ to homotopy pullback squares of simplicial sets, 
and $F(\emptyset)$ is contractible.
Moreover, 
a spectrum $E\in\Spt(\C)$ is fibrant if and only if it is levelwise fibrant and an $\Omega_T$-spectrum. 
There are standard simplicial Quillen adjunctions, 
whose left adjoints preserve weak equivalences
\begin{gather*}
\Sigma_T: \Spc_\bullet(\C)\rightleftarrows \Spc_\bullet(\C):\Omega_T\\
\Sigma^\infty_T : \Spc_\bullet(\C)\rightleftarrows \Spt(\C):
\Omega^\infty_T.
\end{gather*}

Let $(E_n)_{n\geq 0}$ be a levelwise fibrant replacement of $\Sigma^\infty_T(X,x)$,
i.e., 
$E_n$ is a fibrant replacement of $\Sigma^n_T(X,x)$ in $\Spc_\bullet(\C)$, 
and let $E$ be a fibrant replacement of $\Sigma^\infty_T(X,x)$ in $\Spt(\C)$.
A key observation is that filtered colimits in $\Spc_\bullet(\C)$ preserve fibrant objects;
this follows from the above description of fibrant objects and the facts that filtered colimits of simplicial sets preserve Kan complexes, 
homotopy equivalences, 
and homotopy pullback squares (\cite[Corollary 2.16]{DRO}). 
This implies there is a simplicial homotopy equivalence
\[
\Omega^\infty_T E\simeq \colim_{n\to\infty} \Omega^n_T E_n.
\]

Let $\tilde X\in\Spc_\bullet(\C)$ be the simplicial presheaf $(X,x)\vee \Delta^1$ pointed at the free endpoint of $\Delta^1$; 
this is a cofibrant replacement of $(X,x)$ in $\Spc_\bullet(\C)$.
Since $\tilde X\in\Spc_\bullet(\C)$ is $\omega$-compact, 
the following are homotopy equivalences of Kan complexes, 
where $\Map$ denotes the simplicial sets of maps in the above simplicial model categories
\begin{align*}
\Map(\Sigma^\infty_T \tilde X, E)
& \simeq
\Map(\tilde X, \Omega^\infty_T E) \\
& \simeq
\Map(\tilde X,\colim_{n\to\infty}\Omega^n_TE_n) \\
& \simeq
\colim_{n\to\infty}\Map(\tilde X,\Omega^n_TE_n) \\
& \simeq
\colim_{n\to\infty}\Map(\Sigma^n_T\tilde X, E_n).
\end{align*}
The hypothesis that $\Sigma^\infty_T(X,x)$ is weakly contractible means that the weak equivalence $\Sigma^\infty_T\tilde X\stackrel\sim\to E$ and the zero
map $\Sigma^\infty_T\tilde X\to *\to E$ are in the same connected component of the Kan complex $\Map(\Sigma^\infty_T\tilde X,E)$.
Since $\pi_0$ preserves filtered colimits of simplicial sets, 
there exists $n\geq 0$ such that the weak equivalence $\Sigma^n_T\tilde X\stackrel\sim\to E_n$ and the zero map $\Sigma^n_T\tilde X\to *\to E_n$ belong to the 
same connected component of $\Map(\Sigma^n_T\tilde X, E_n)$. 
In other words, 
$\Sigma^n_T\tilde X\simeq\Sigma^n_T(X,x)$ is $\A^1$-contractible.
\end{proof}

\begin{cor}\label{cor:Trivial-EC}
Let $X$ be a Koras--Russell threefold of the first or second kind, 
and let $B\C^{\times} \simeq \C\P^{\infty}$ denote the motivic classifying space of $\C^{\times}$ (cf.~\cite[\S~4]{MV}). 
Then the projection map 
\begin{equation}\label{eqn:Triv-EC-0}
\Sigma^{\infty}_T(X \stackrel{\C^{\times}}{\times} E\C^{\times})_{+} 
\to 
\Sigma^{\infty}_T(B\C^{\times})_{+}
\end{equation}
is an isomorphism.
\end{cor}
\begin{proof}
By definition, 
$\Sigma^{\infty}_T(X \stackrel{\C^{\times}}{\times} E\C^{\times})_{+}$ is the colimit of the compact objects $\Sigma^{\infty}_T(X \stackrel{\C^{\times}}{\times} U_n)_{+}$ for $n\geq 1$, 
where $U_n=\A^n \smallsetminus \{0\}$ with free $\C^{\times}$-action given by scalar multiplication. 
Hence, it suffices to show there is an isomorphism
$$
\Sigma^{\infty}_T(X \stackrel{\C^{\times}}{\times} U_n)_{+} 
\to
\Sigma^{\infty}_T(\C\P^{n-1})_{+}
$$
for each $n \ge 1$.
Since $U_n \to \C\P^{n-1}$ is Zariski locally trivial, 
in fact a principal $\C^{\times}$-bundle, 
so is the map $X \stackrel{\C^{\times}}{\times} U_n \to \C\P^{n-1}$ with fiber $X$.
Using the Mayer--Vietoris exact triangles
\[
\Sigma^{\infty}_T(U \cap V)_{+} \to \Sigma^{\infty}_T(U)_{+} \oplus
\Sigma^{\infty}_T(V)_{+} \to \Sigma^{\infty}_T(U \cup V)_{+} \to
\Sigma^{\infty}_T(U \cap V)_{+}[1]
\]
and induction on the length of an affine open cover of $\C\P^{n-1}$, 
we are reduced to showing that for any smooth complex affine variety $Y$,
the projection $\Sigma^{\infty}_T(X \times Y)_{+} \to\Sigma^{\infty}_T(Y)_{+}$ is an isomorphism. 
This follows from the isomorphisms $\Sigma^{\infty}_T(X)_{+} \xrightarrow{\simeq}\1$ (cf.~\thmref{thm:Main-KR}), 
and $\Sigma^{\infty}_T(X \times Y)_{+} \simeq \Sigma^{\infty}_T(X)_{+} \wedge\Sigma^{\infty}_T(Y)_{+}$.
\end{proof}

\begin{remk}
	Given a spectrum $E\in\SH(\C)$ and a smooth complex variety $X$ with $\C^\times$-action,
	the Borel $\C^\times$-equivariant $E$-cohomology groups of $X$ can be defined as
	\[
	E^{*,*}_{\C^\times}(X)=E^{*,*}(X \stackrel{\C^{\times}}{\times} E\C^\times).
	\]
	Special cases include the equivariant Chow groups \cite{EdidinGraham} and Borel equivariant $K$-theory \cite{Krishna}.
	If $X$ is a Koras--Russell threefold of the first or second kind, Corollary~\ref{cor:Trivial-EC} shows that the Borel $\C^\times$-equivariant $E$-cohomology of $X$ is trivial, 
        in the sense that the map $X\to\Spec\C$ induces an isomorphism
	\[
	E^{*,*}_{\C^\times}(X)\simeq E^{*,*}_{\C^\times}(\Spec\C).
	\]
	This implies that the other two terms in the Milnor exact sequence (see \cite[Proposition~7.3.2]{Hovey})
	\[
	0 \to {{\underset{n}\varprojlim}}^1 \ E^{a-1, b}(X \stackrel{\C^{\times}}{\times} U_n) \to
	E^{a,b}(X \stackrel{\C^{\times}}{\times} E\C^{\times}) \to {\underset{n}\varprojlim} \ E^{a,b}(X \stackrel{\C^{\times}}{\times} U_n) 
	\to 0
	\]
	are trivial in the same sense (the last term is sometimes taken as the definition of Borel equivariant cohomology).
\end{remk}

\begin{remk}
Koras--Russell threefolds were initially studied over $\C$ and with $\C^{\times}$-actions.
All results in this paper, 
as well as the proofs, 
are valid for algebraically closed fields of characteristic zero and $\G_m$.
\end{remk}

\begin{remk}
Recently, 
Dubouloz and Fasel have announced a proof that Koras-Russell threefolds of the first kind are unstably $\A^1$-contractible \cite{DF}.
\end{remk}

{\bf Acknowledgments.} The authors thank Aravind Asok for very interesting and helpful discussions which motivated this paper.  
We also gratefully acknowledge the hospitality and support during the Motivic Homotopy Theory semester organized by Marc Levine at Universit{\"a}t Duisburg--Essen,
where this work was carried out.
The third author acknowledges support from the RCN project Special Geometries, no.~239015.
The authors also thank the referee for a careful reading of the paper and valuable comments.

\begin{footnotesize}




\vspace{0.1in}

\begin{center}
Department of Mathematics, Massachusetts Institute of Technology, Cambridge, MA, USA.\\
e-mail: hoyois@mit.edu
\end{center}
\begin{center}
School of Mathematics, Tata Institute of Fundamental Research, Homi Bhabha Road, Colaba, Mumbai, India.\\
e-mail: amal@math.tifr.res.in
\end{center}
\begin{center}
Department of Mathematics, University of Oslo, Norway.\\
e-mail: paularne@math.uio.no
\end{center}
\end{footnotesize}
\end{document}